\documentclass[11pt]{amsart}

\usepackage{amsmath, amsthm, amssymb}
\usepackage{lmodern}
\usepackage{palatino}                       
\usepackage[bigdelims,vvarbb]{newpxmath}    
\usepackage[scaled=0.95]{inconsolata}       
\usepackage[T1]{fontenc}                    

\usepackage{comment}

\usepackage[abs]{overpic}		  
\usepackage{import}
\usepackage{transparent}

\usepackage{xcolor}
\definecolor{lred}{RGB}{226, 106, 106}
\definecolor{nred}{RGB}{237, 28, 36}
\definecolor{lblue}{RGB}{52, 152, 219}
\definecolor{nblue}{RGB}{0, 174, 239}
\definecolor{lyellow}{RGB}{232, 197, 91}
\definecolor{dgreen}{RGB}{0, 148, 68}
\definecolor{l1yellow}{RGB}{217, 224, 33}
\definecolor{lgrey}{RGB}{179, 179, 179}
\definecolor{indigo}{rgb}{0.29, 0.0, 0.51}  
\usepackage[colorlinks, urlcolor=indigo, linkcolor=indigo, citecolor=indigo]{hyperref}

\usepackage[hcentering, vcentering, total={5.9in, 8.2in}]{geometry}  

\usepackage{tikz-cd}
\usetikzlibrary{cd}
\tikzset{
  symbol/.style={
    draw=none,
    every to/.append style={
      edge node={node [sloped, allow upside down, auto=false]{$#1$}}
    },
  },
}
\usetikzlibrary{trees}

\theoremstyle{plain}

\newtheorem{theorem}{Theorem}

\newtheorem{proposition}[theorem]{Proposition}
\newtheorem{lemma}[theorem]{Lemma}

\newtheorem{claim}{Claim}

\theoremstyle{definition}
\newtheorem{definition}[theorem]{Definition}

\theoremstyle{remark}
\newtheorem{remark}[theorem]{Remark}

\numberwithin{theorem}{section}

\newcommand{\Q}{\mathbb{Q}}           
\newcommand{\Z}{\mathbb{Z}}           

\makeatletter
\newcommand*\bigcdot{\mathpalette\bigcdot@{0.6}}
\newcommand*\bigcdot@[2]{\mathbin{\vcenter{\hbox{\scalebox{#2}{$\m@th#1\bullet$}}}}}
\makeatother

\newcommand{\Pushoff}{{\def\svgwidth{1,6ex}\,\,
\begingroup%
  \makeatletter%
  \providecommand\color[2][]{%
    \errmessage{(Inkscape) Color is used for the text in Inkscape, but the package 'color.sty' is not loaded}%
    \renewcommand\color[2][]{}%
  }%
  \providecommand\transparent[1]{%
    \errmessage{(Inkscape) Transparency is used (non-zero) for the text in Inkscape, but the package 'transparent.sty' is not loaded}%
    \renewcommand\transparent[1]{}%
  }%
  \providecommand\rotatebox[2]{#2}%
  \ifx\svgwidth\undefined%
    \setlength{\unitlength}{49.54732089bp}%
    \ifx\svgscale\undefined%
      \relax%
    \else%
      \setlength{\unitlength}{\unitlength * \real{\svgscale}}%
    \fi%
  \else%
    \setlength{\unitlength}{\svgwidth}%
  \fi%
  \global\let\svgwidth\undefined%
  \global\let\svgscale\undefined%
  \makeatother%
  \begin{picture}(1,1.00059637)%
    \put(0,0){\includegraphics[width=\unitlength,page=1]{PushOff.pdf}}%
  \end{picture}%
\endgroup%
\,\,}} 

\DeclareFontFamily{U} {cmr}{}
\DeclareFontShape{U}{cmr}{m}{n}{
  <-6> cmr5
  <6-7> cmr6
  <7-8> cmr7
  <8-9> cmr8
  <9-10> cmr9
  <10-12> cmr8
  <12-> cmr9}{}
\DeclareSymbolFont{Xcmr} {U} {cmr}{m}{n}

\title{Contact surgery distance}

\author{Marc Kegel}
\address{Universidad de Sevilla, Dpto.\ de Álgebra,
Avda.\ Reina Mercedes s/n,
41012 Sevilla}
\email{mkegel@us.es, kegelmarc87@gmail.com}

\author{Isacco Nonino}
\address{School of Mathematics and Statistics, University of Glasgow}
\email{2754452N@student.gla.ac.uk} 

\author{Monika Yadav}
\address{Indian Statistical Institute, Kolkata, India}
\email{monika.yadav9413@gmail.com}

\date{\today}

\begin{document}

\begin{abstract}
In this article, we define the \emph{contact surgery distance} of two contact 3-manifolds $(M,\xi)$ and $(M',\xi')$ as the minimal number of contact surgeries needed to obtain $(M,\xi)$ from $(M',\xi')$. Our main result states that the contact surgery distance between two contact $3$-manifolds is at most $5$ larger than the topological surgery distance between the underlying smooth manifolds.
As a byproduct of our proof, we classify the rational homology $3$-spheres on which the $d_3$-invariant of a $2$-plane field already determines its $\Gamma$-invariant and Euler class. 
\end{abstract}

\makeatletter
\@namedef{subjclassname@2020}{%
  \textup{2020} Mathematics Subject Classification}
\makeatother

\subjclass[2020]{53D35; 53D10, 57K10, 57R65, 57K33} 

\maketitle

\section{Introduction}

A seminal result in 3-dimensional contact topology, due to Ding and Geiges, states that every oriented, closed, contact 3-manifold with coorientable contact structure can be obtained by contact $(\pm1)$-surgery on a Legendrian link $\mathcal{L}$ in the standard contact $3$-sphere $(S^3,\xi_{std})$ \cite{DingGeiges}. Thus a natural complexity notion for a given contact 3-manifold $(M,\xi)$ is its \emph{contact surgery number} $cs_{\pm1}(M,\xi)$, i.e.\ the minimal number of components of a Legendrian link $\mathcal{L}$ in $(S^3,\xi_{std})$ needed to obtain $(M,\xi)$ via contact $(\pm1)$-surgery along $\mathcal{L}$. The study of contact surgery numbers was initiated in \cite{KegelEtnyreOnaran,ChatterjeeKegel,KegelYadav}. 

Since we can reverse a contact $(\pm1)$-surgery by a contact $(\mp1)$-surgery it follows also that any contact $3$-manifold can be obtained from any other contact $3$-manifold by contact $(\pm1)$-surgery. We define the \emph{contact surgery distance} $cs_{\pm1}((M,\xi),(N,\xi'))$ of two contact $3$-manifolds $(M,\xi)$ and $(N,\xi')$ as the minimal number of components of a Legendrian link $\mathcal{L}$ in $(M,\xi)$ such that $(N,\xi')$ is obtained via contact $(\pm1)$-surgery along $\mathcal{L}$. 

We emphasise that the contact surgery distance $cs_{\pm1}((M,\xi),(N,\xi'))$ is bounded from below by the \emph{integral topological surgery distance} $s_\Z(M,N)$, i.e.\ the minimal number of integral smooth surgeries from $M$ to $N$. (We refer to \cite{Ichihara-Saito} or Section~\ref{sc:surgery_distance} for a more detailed discussion, for example, why $s_\Z$ is well-defined.) The following is our main result.

\begin{theorem}\label{thm:contact_surgery_distance}
Let $(M,\xi)$ and $(N,\xi')$ be two closed, oriented, contact 3-manifolds. Then 
$$s_{\Z}(M,N)\leq cs_{\pm1}\big((M,\xi),(N,\xi')\big)\leq s_{\Z}(M,N)+5.$$
\end{theorem}

There are examples of contact 3-spheres for which $s_\Z$ is different from $cs_{\pm1}$, see for example \cite{KegelEtnyreOnaran}. In fact, only the standard tight contact structure $\xi_{std}$ on $S^3$ has contact surgery number 0. Hence, we have an infinite collection of examples of contact 3-manifolds for which the contact surgery distance is strictly larger than the topological surgery distance. It remains open if the upper bound from Theorem~\ref{thm:contact_surgery_distance} is sharp. 

In Section~\ref{sec:lens_spaces}, we study contact surgery distance of contact lens spaces. In particular, we use results of Plamenevskaya~\cite{Plamenevskaya} to obtain lower bounds that do not come from the underlying smooth manifolds.

\begin{remark}
    As for the contact surgery numbers, one could ask for the surgery coefficients to be more general; for example, one can consider all non-vanishing rational surgery coefficients. In Section~\ref{sc:surgery_distance} we define such surgery 'distances', but since one cannot reverse a general rational contact surgery by a single contact surgery, this new notion is not symmetric anymore. We refer to Section~\ref{sc:surgery_distance} for details. Nevertheless, we are able to prove a generalisation of Theorem~\ref{thm:contact_surgery_distance} for general contact surgery coefficients, see Theorem \ref{thm:maingeneral}.
\end{remark}
In order to prove Theorem~\ref{thm:contact_surgery_distance}, we analyse how the homotopical invariants of a contact structure change under a Lutz twist. Recall that the overtwisted contact structures on a rational homology $3$-sphere are determined by the underlying tangential 2-plane fields~\cite{Eliashberg_OT}, which in turn are completely classified by the two homotopical invariants: Gompf’s $\Gamma$-invariant (whose double gives the Euler class) and the $d_3$-invariant~\cite{Gompf_Stein}. On the other hand, Lutz twisting is an operation performed on a transverse knot $T$ in a contact $3$-manifold $M$ that does not change the underlying smooth manifold but changes the underlying $2$-plane field (and makes the contact structure always overtwisted). We refer, for example, to~\cite{Geiges} for a detailed discussion on Lutz twists. For the proof of Theorem~\ref{thm:contact_surgery_distance}, we obtain the following result, which might be of independent interest and generalises similar results obtained in~\cite{Lutz_thesis,DingGeigesStipsciz,Geiges} for the relative obstruction classes of two contact structures.

\begin{proposition}\label{prop:lutz_twist_rationally_null_homologous} 
    Let $(M,\xi)$ be a contact manifold with torsion Euler class $e(\xi)$ and let $T$ a rationally null-homologous, transverse knot in $M$ with rational self-linking number $sl_{\mathbb{Q}}(T)$. Let $\xi_T$ be the contact structure on $M$ obtained by performing a Lutz twist along $T$. Then their $\Gamma$-invariants are related by
        \begin{align*}
        \Gamma(\xi_T,\mathfrak{s})-\Gamma(\xi,\mathfrak{s})=&-[T]\in H_1(M;\Z).
    \end{align*}
    If in addition $M$ is a rational homology $3$-sphere, then the difference of the $d_3$-invariants is
    \begin{align*}
        d_3(\xi_T)-d_3(\xi)=&-sl_{\mathbb{Q}}(T)\in\Z.
    \end{align*}
\end{proposition}
As another application of Proposition~\ref{prop:lutz_twist_rationally_null_homologous}, we completely answer a question raised in~\cite{KegelYadav} and determine the rational homology $3$-spheres on which the $d_3$-invariant of a $2$-plane field already determines its $\Gamma$-invariant and Euler class. 

\begin{theorem}\label{thm:d_3determinesinvariants}
    Let $M$ be a rational homology $3$-sphere. Then the $d_3$-invariant of every tangential $2$-plane field $\xi$ on $M$ determines the $\Gamma$-invariant if and only if $H_1(M;\mathbb{Z})$ is isomorphic to $0$ or $\Z_2$.
\end{theorem}

\subsection*{Conventions}
Throughout this paper, all $3$-manifolds are assumed to be connected, clos\-ed, and oriented and all contact structures are assumed to be positive and coorientable. Since a contact surgery diagram determines a contact manifold only up to contactomorphism, this will be the notion we use to distinguish them, not isotopy. A contact manifold defined via a contact surgery on a Legendrian knot is up to contactomorphism independent of the orientation of the Legendrian knot. Nevertheless, we will often endow Legendrian links with orientations to simplify computations, so that, for example, the rotation numbers are well-defined. Legendrian knots in $(S^3,\xi_{std})$ are always presented in their front projection. We will write $t$ and $r$ for the Thurston--Bennequin invariant and the rotation number of an oriented Legendrian knot in $(S^3, \xi_{std})$. If the knot is in a general manifold $(M,\xi)$, we will use the notation $tb$ and $rot$.  We also normalize the $d_3$-invariant such that $d_3(S^3,\xi_{std})=0$, following conventions used for example in \cite{KegelEtnyreOnaran,Kegel_Onaran,Casals_Etnyre_Kegel,ChatterjeeKegel,KegelYadav}. Note that this differs from the original sources~\cite{Gompf_Stein,Ding_Geiges_Stipsciz_surgery_diagrams}.

\subsection*{Acknowledgments} 
The first author thanks Misha Schmalian for teaching him the basics about quadratic reciprocity, which we have used in Section~\ref{sec:lens_spaces}.
The second author thanks the University of Glasgow for providing travel support and the Humboldt University of Berlin for the warm hospitality. 

\subsection*{Grant support}
MK is supported by the DFG, German Research Foundation, (Project: 561898308); by a Ram\'on y Cajal grant (RYC2023-043251-I) and PID2024-157173NB-I00 funded by MCIN/AEI/10.13039/501100011033, by ESF+, and by FEDER, EU; and by a VII Plan Propio de Investigación y Transferencia (SOL2025-36103) of the University of Sevilla. This project started when IN and MY visited MK at the Humboldt University Berlin. The research visit of MY to Humboldt University of Berlin was supported by the BMS Math+ Hanna Neumann Fellowship. MY gratefully acknowledges the support of BMS Math+ and the warm hospitality of Humboldt University. IN is supported by EPSRC scholarship.

\section{Contact Dehn surgery}

In this section, we provide the necessary background on contact Dehn surgery. For more extended discussion we refer for example to \cite{Gompf_Stein,DingGeiges,Ding_Geiges_Stipsciz_surgery_diagrams,Ozbagci_Stipsicz_book,Geiges,Kegel-Durst,Kegel_thesis,Kegel_knot_complement,KegelEtnyreOnaran,Casals_Etnyre_Kegel}.
Let $K$ be a Legendrian knot in a contact $3$-manifold $(M,\xi)$. A \emph{contact Dehn surgery} along $K$ is removing a standard tubular neighborhood of $K$ and attaching a solid torus $\mathbb{S}^1\times D^2$ via a diffeomorphism that maps $\{pt\} \times \partial D^2$ to $p\mu+q\lambda_c$. Here $\mu$ is the meridian of $K$ and the \emph{contact longitude} $\lambda_c$ is a Legendrian knot obtained by pushing $K$ in the Reeb direction. The resulting manifold carries natural contact structures that coincide with $\xi$ outside the removed tubular neighborhood and are tight on the glued-in solid torus. It was shown by Honda \cite{Honda} and Giroux~\cite{Giroux_lens} that such structures always exist and that it is unique if $p=\pm 1$. For general rational contact surgery coefficients $\frac{p}{q}$ the resulting structure is not unique and there are finitely many different choices depending on the \emph{continued fraction expansion} of $\frac{p}{q}$ (see \cite{DingGeiges}). The following result was already mentioned in the introduction.

\begin{theorem}[Ding--Geiges \cite{DingGeiges}]\label{thm:lik-wallace-contact}
    Let $(M,\xi)$ be a contact 3-manifold. Then $(M,\xi)$ can be obtained by rational contact Dehn surgery along a Legendrian link in $(S^3,\xi_{std})$. Moreover, one can assume all contact surgery coefficients to be of the form $(\pm1)$.\qed
\end{theorem}

The following lemma shows that a contact $(+1)$-surgery is inverse to a contact $(-1)$-surgery.

\begin{lemma}[Cancellation Lemma \cite{DingGeiges}]\label{lem:cancelation}
 Let $K$ be a Legendrian knot in $(M,\xi)$ and $(M',\xi')$ be the result of contact $(\pm1)$-surgery along $K$. Then the core of the newly glued-in solid torus is again a Legendrian knot, the dual knot $K^*$, and contact $(\mp1)$-surgery along $K^*$ produces again $(M,\xi)$.\qed
\end{lemma}

It is rather important for our work to understand when a given Legendrian or transverse knot in a contact $3$-manifold can be presented in a rational contact surgery diagram of that manifold. In general, this is not possible, see \cite{Kegel_knot_complement} for a concrete example. However, in certain situations this holds true.

\begin{lemma}[\cite{Kegel_knot_complement}]\label{lem:knot_in_complement}
 Let $\mathcal{L}= L_1\cup\cdots  \cup L_n \subset  (S^3
, \xi_{st})$ be a contact surgery diagram with
$(\pm1)$-contact surgery coefficients, $i = 1, \cdots , n$ of a contact manifold $(M, \xi)$. Then
any Legendrian or transverse knot $K$ in $(M,\xi)$ can be represented by a Legendrian knot in the
exterior of $L$.\qed
\end{lemma}

\section{The homotopical invariants of a contact structure}\label{sec:invariants}

In this section, we present the necessary background on the homotopical invariants of the underlying tangential $2$-plane field of a contact structure. It is known that a tangential $2$-plane field $\xi$ on a rational homology sphere $M$ is completely determined (up to homotopy) by the $d_3$-invariant and Gompf's $\Gamma$-invariant~\cite{Gompf_Stein,DingGeigesStipsciz}. 
Roughly speaking, the $\Gamma$-invariant is a refinement of the Euler class (since its double is the Euler class) and encodes $\xi$ on the $2$-skeleton of $M$, while the $d_3$-invariant specifies $\xi$ on the $3$-cell. In reality the $\Gamma$-invariant only becomes relevant whenever there is 2-torsion, if not, one could just use the Euler class, which has the advantage of not needing to introduce \emph{spin structures}. In this work, we will need to use the $\Gamma$-invariant since we usually have cases where $2$-torsion can appear.

We briefly recall the following useful lemma to compute the $d_3$-invariant from a contact $(\pm1)$-surgery diagram~\cite{Ding_Geiges_Stipsciz_surgery_diagrams,Kegel-Durst}.

 \begin{lemma}\label{lem:d3}
     Let $L=L_1\cup\cdots\cup L_k$ be an oriented Legendrian link in $(S^3,\xi_{std})$ and let $(M,\xi)$ be the contact manifold obtained by contact $\left(\pm1\right)$-surgeries along $L$. Let $t_i$ and $r_i$ be the Thurston--Bennequin invariant and rotation number of $L_i$ for all $i=1,\ldots,k$. Let $l_{ij}$ be the linking number of $L_i$ with $L_j$ and let ${p_i}=t_i\pm 1$ be the topological surgery coefficient of $L_i$. We define the linking matrix $Q$ by
\begin{align*}
  Q=\begin{pmatrix}
     p_1&l_{12}&\cdots& l_{1k}\\
     l_{12}&p_2&\cdots &l_{2k}\\
     \vdots&&\ddots&\vdots\\
     l_{1k}&&&p_k
  \end{pmatrix} . 
\end{align*}
\begin{enumerate}
    \item The first homology $H_1(M;\Z)$ is presented  by the abelian group $\langle\mu_i| Q\mu^T=0\rangle$, where $\mu=(\mu_1,\ldots,\mu_k)$ is the vector of meridians $\mu_i$ of $L_i$.
    \item If there exists a rational solution $b \in \Q^k$ of $Qb = r$, where $r=(r_1,\ldots,r_k)^T$, the $d_3$-invariant is well defined and is computed as
    $$\frac{1}{4}\{\langle r,b\rangle-3\sigma(Q)-2\operatorname{rk}(Q)-2\}-\frac{1}{2}+q$$
    where $\langle \cdot,\cdot\rangle$ simply denotes the vector product, $\sigma(Q)$ and $\operatorname{rk}$ denote the signature and rank of $Q$, and $q$ is the number of positive contact surgeries in the diagram. \qed
\end{enumerate}   
 \end{lemma}

Next, we discuss the $\Gamma$-invariant of a contact manifold $(M, \xi)$. Let $\mathfrak{s}$ be a spin structure on $M$, then Gompf~\cite{Gompf_Stein} defines an invariant $\Gamma(\xi,\mathfrak{s})\in H_1(M)$, that depends only on the tangential $2$-plane field $\xi$ and the spin structure $\mathfrak{s}$. To explain how to compute $\Gamma$-invariant from a contact $(\pm1)$-surgery diagram, we first recall how spin structures are represented in surgery diagrams~\cite{Gompf_Stipsicz}.

Let $\mathcal{L}= L_1 \cup \dots \cup L_k$ be an integer surgery diagram of a smooth $3$-manifold $M$ along an oriented link $\mathcal{L}$, where $l_{ii}$ denotes the the framing of $L_i$ measured relative to the Seifert framing. A sublink $\mathcal{L}_J=(L_j)_{j \in J}$, for some subset $J \subseteq \{1, 2, \dots, k\}$, is called a \textit{characteristic sublink} if  
\[
l_{ii} \equiv \sum_{j \in J} l_{ij} \pmod{2}
\]  
for every component $L_i$ of $\mathcal{L}$. The characteristic sublinks of $\mathcal{L}$ are in natural bijection with the spin structures on $M$~\cite{Gompf_Stipsicz}. Thus, each spin structure on $M$ can be represented by a characteristic sublink of $\mathcal{L}$.
The following lemma from~\cite{KegelEtnyreOnaran} explains how to compute the $\Gamma$-invariant from a contact $(\pm1)$-surgery diagram. 

\begin{lemma}  \label{lem:Gamma}
Let \( \mathcal{L} = L_1 \cup \dots \cup L_k \) be a Legendrian link in \( (S^3, \xi_{std}) \), and let \( (M, \xi) \) be the contact manifold obtained by performing contact $(\pm 1)$-surgery along \( \mathcal{L} \). 
Let $\mathcal{L}_J$ be a characteristic sublink describing a spin structure \( \mathfrak{s} \) on \( M \). Then the \( \Gamma \)-invariant satisfies  
\[
\pushQED{\qed} 
\Gamma(\xi, \mathfrak{s}) = \frac{1}{2} \left( \sum_{i=1}^k r_i \mu_i + \sum_{j \in J} (Q\mu)_j \right) \in H_1(M).\qedhere
\popQED
\]  
\end{lemma}  

To compare $\Gamma$-invariants of contact structures described by different contact surgery diagrams of the same underlying smooth $3$-manifold, we need to understand how the spin structures in these surgery diagrams are related. The following lemma, which can also be derived from ~\cite{Gompf_Stipsicz}, tells us how characteristic sublinks change under Kirby moves. 
 
\begin{lemma} \label{lem:sublink_modifications}
    Let $\mathcal{L} = L_1 \cup \dots \cup L_k$ be a smooth oriented integer surgery link with characteristic sublink $\mathcal{L}_J$ representing a spin structure $\mathfrak{s}$. Then the following modifications of the characteristic sublink under Kirby moves preserve the spin structure $\mathfrak{s}$.
    \begin{itemize}
      \item \textbf{Blow up/down:} Let $U$ be an unknot with surgery coefficient $\pm 1$ which is added to $\mathcal{L}$ under a blow up move. Then $U$ gets added to the characteristic sublink if and only if $$\sum_{j\in J}lk(U,L_j)=0\, \pmod{2}.$$ The other components of the characteristic sublink stay unchanged. Under the blow down move, we remove an unknot with $\pm1$ coefficients from $\mathcal{L}$ and the other components of the characteristic sublink stay the same.
      \item \textbf{Handle slide:}
      If we slide the component $L_i$ over the component $L_k$ then $L_k$ changes its membership status in the characteristic sublink if and only if $L_i$ is in the characteristic sublink. All other components of the characteristic sublink stay the same.
      \item \textbf{Rolfsen twist:} If $L_k$ is a $0$-framed unknot and we perform an $n$-fold Rolfsen twist on $L_k$, then the resulting surgery diagram is again an integer surgery diagram. All components of the characteristic sublink remain unchanged, except for a possible change of the unknot $L_k$ which is being twisted. $L_k$ changes the membership status to the characteristic sublink if and only if 
     $$n\left(1+\sum_{j\in J-\{k\}}lk(L_k,L_j)\right)=1\, \pmod{2}.$$
     \item \textbf{Slam dunk:} Let $\overline{\mathcal{L}}=\mathcal{L}\cup K \cup
     U$, where $U$ is a $0$-framed unknot and $K$ is an $n$-framed knot as shown in Figure \ref{fig:charsublink} (with $n \in \mathbb{Z}$).  Performing a slam dunk move  on this pair will remove $U$ and $K$ from the surgery diagram and we get the surgery diagram $\mathcal{L}$. The rest of the components in the surgery diagram that are part of the characteristic sublink remain part of the characteristic sublink.
     \item \textbf{Inverse Slam dunk:}
     We describe the modifications of the characteristic sublink under the inverse of the slam dunk move described above. In this situation $K$ never gets added to the characteristic sublink and $U$ gets added to the characteristic sublink if and only if the quantity $n-lk(K,\mathcal{L}_J)$ is an odd integer, where $\mathcal{L}_J$ is the characteristic sublink of $\mathcal{L}$. 
\end{itemize}     
\end{lemma}

\begin{proof}
     Handle slides, blow ups/downs, and Rolfsen twists are discussed in \cite{Gompf_Stipsicz}. Thus we only prove the case of slam dunks.
     
      We start with a surgery diagram $\mathcal{L}$ of $M$ and a characteristic sublink $\mathcal{L}_J$ in it corresponding to some spin structure $\mathfrak{s}$. Let $\overline{\mathcal{L}}$ be a surgery diagram obtained by adding the canceling pair $K\cup U$ to $\mathcal{L}$, where $K$ has framing $n\in \mathbb{Z}$ and $U$ is a $0$-framed unknot (see Figure \ref{fig:charsublink}). Note that here $K$ might be non trivially linked with $\mathcal{L}$ and $U$ is only linked with $K$ as shown in Figure \ref{fig:charsublink}. In order to see how the characteristic sublink changes we write down a sequence of blow ups/downs to go between $\mathcal{L}$ and $\overline{\mathcal{L}}$ and track the characteristic sublink $\mathcal{L}_J$ under these moves using the first point of this lemma. The first diagram in the Figure \ref{fig:charsublink} shows $K, U$, and the components of $\mathcal{L}$ which might be linked with $K$ are shown in green color. We can turn $K$ into an unknot by blowing up unknots $O_1,\hdots, O_k$ (where $k$ is some positive integer) with framing $x_1,\hdots,x_k\in \{\pm 1\}$ near  crossings in a diagram of $K$ such that linking number of these unknots with $K$ is $0$. The number $k$ is determined by the unknotting number of $K$, where each crossing change corresponds to a $(\pm 1)$-unknot blow up. The unknots $O_i$s are shown in yellow color in Figure \ref{fig:charsublink}. Since the linking number of the blown up unknots $O_i$ with $K$ is $0$ and they do not link with $\mathcal{L}$ at all, the framing of $K$ and $\mathcal{L}$ stay unchanged. We can hence assume that $K$ is an unknot with framing $n \neq 0$, the case for $ n=0$ can be done similarly. We blow up $(|n|-1)$-many unknots $U_1,\hdots, U_{|n|-1}$ with framing $\mp1$, where the sign is meant to be negative if $n$ is positive and vice-versa (the purple unknots in Figure \ref{fig:charsublink}). The new framing of $K$ is $\pm1$ and we have $(|n|-1)$ unknots $U_i$ with $|lk(K,U_i)|=1$. Blow down the unknot $K$ to get the surgery diagram where the framings of the components of $L$ may change depending on the their linking number with $K$. However the framings of the yellow unknots do not change. Then we blow down the unknot $U$ to get the surgery diagram $\mathcal{L}$ union some unknots with framings in $\{+1,-1\}$ which are not linked at all with $\mathcal{L}$. Finally, blow down the extra unknots to get $\mathcal{L}$. 
      
      Now, we trace the characteristic sublink under these moves. We start with the last surgery diagram $\mathcal{L}$ in Figure \ref{fig:charsublink}. We put a star on the the components of $\mathcal{L}$ which are part of the characteristic sublink $\mathcal{L}_J$. The corresponding characteristic sublink on the first diagram, i.e.\ $\overline{\mathcal{L}}$ will depend on the parity of $n-lk(K,\mathcal{L}_J)$. Hence, in the last picture we use blue stars for $n-lk(K,\mathcal{L}_J)$ odd and pink stars for the even case. Note that, in  Figure \ref{fig:charsublink} we didn't put any stars on the green knots which are components of $\mathcal{L}$ because the characteristic sublink restricted to $\mathcal{L}$ never changes in this sequence of blow ups/downs. In the second to last diagram all yellow and purple unknots get added to the characteristic sublink as their linking number with $\mathcal{L}$ is zero (see the first bullet point of the Lemma). After an isotopy, we blow up the unknot $U'$ (middle row, middle diagram). Note that the linking number of $U'$ with the characteristic sublink up to sign is $ lk(U',\mathcal{L}_J)+(|n|-1)$ since $lk(U',O_i)=0$ for all $i$ and $|lk(U',U_i)|=1$ for all $1\leq i\leq |n|-1$. Also note that $lk(K,\mathcal{L}_J)=lk(U',\mathcal{L}_J)$. Therefore, $U'$ gets a star if and only if $n-lk(K,\mathcal{L}_J)$ is odd, that is $U'$ only gets a blue star. Then we blow up a blue unknot (right most diagram in the first row).
      \begin{claim}
         The blue unknot $K'$ never gets a star as its linking number with the characteristic sublink is odd.
      \end{claim} 
      \begin{proof}[Proof of the claim]
          We explain this claim briefly. Let $l$ be the linking number of $U'$ with the characteristic sublink $ \mathcal{L}_J$. We consider two cases depending on whether $U'$ is added to the characteristic sublink or not.
          If $U'$ gets added to $\mathcal{L}_J$ then $l$ is even. The linking number of $K'$ with the characteristic sublink is now $l \pm 1$, which is odd. This is because $lk(K',U')= \pm 1$. So $K'$ does not get a star. If $U'$ is not in the characteristic sublink then $l$ is odd. The linking number of $K'$ with the characteristic sublink is now $l$, and hence it is odd as well. This concludes the proof of the claim.
      \end{proof} 
      
      We blow down the purple $(|n|-1)$ unknots and yellow unknots to get $\overline{\mathcal{L}}$.
      Thus, we showed that that $K$ never gets added to the characteristic sublink and $U$ gets added to the characteristic sublink if and only if $n-lk(K,\mathcal{L}_J)$ is odd.
    \end{proof} 

          \begin{figure}
          \centering
          \includegraphics[width=0.9\linewidth]{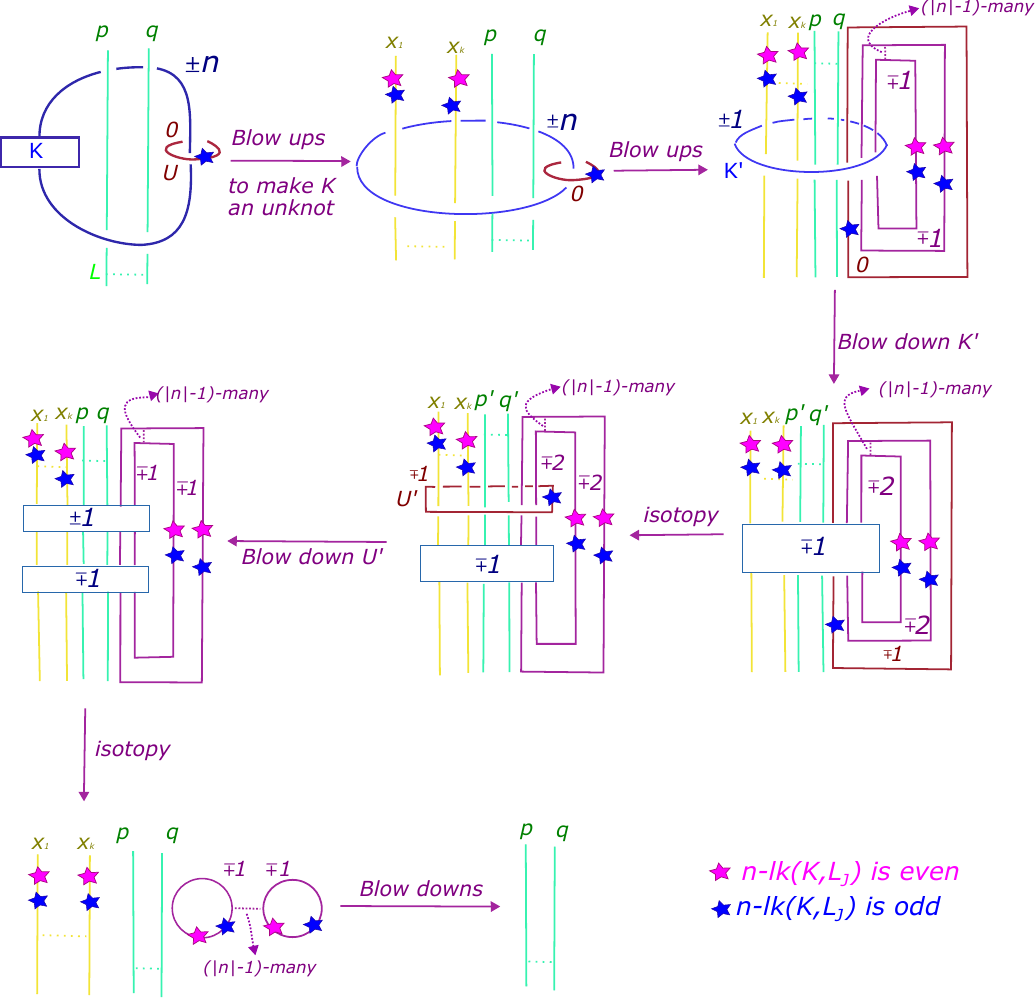}
          \caption{Sequence of blow ups/downs connecting $\overline{\mathcal{L}}$ to $\mathcal{L}$. Here $x_i, p,q$ are framing coefficients and $x_i\in \{\pm 1\}$.}
          \label{fig:charsublink}
      \end{figure}
   
 Later we will also use the following lemma saying that the set of the homotopical invariants of a tangential $2$-plane field on a manifold just depend on its homology.
 
\begin{lemma}[\cite{Gompf_Stein}]\label{lem:homologydependenceonly}
    Let $M$ and $N$ be two orientable, connected and closed 3-manifolds with $H_1(M;\mathbb{Z})\cong H_1(N;\mathbb{Z})$. Then for a fixed spin structure $\mathfrak{s}$ on $M$, there exists a spin structure $\mathfrak{s}'$ on $N$  such the set $\left\{(d_3(\xi),\Gamma(\xi;\mathfrak{s}))\:|\: \xi\in \{\text {tangential $2$-plane field on} M\}\right\}$ is in $1 \colon 1$ correspondence with the set $\left\{(d_3(\xi'),\Gamma(\xi';\mathfrak{s}'))\:|\: \xi'\in \{\text {tangential $2$-plane fields on} N\}\right\}$. 
\end{lemma}

\begin{proof}
    This is in reality just a reformulation of the classification theorem for tangential 2-plane fields in \cite{Gompf_Stein}. A homotopy class of a tangential 2-plane field is determined by the pair $(d_3, \Gamma)$. Since such invariants ultimately depend on the first homology of the 3-manifold, if $H_1(M;\mathbb{Z})$ and $H_1(N;\mathbb{Z})$ are isomorphic, there exists a bijection between the possible values for $(d_3,\Gamma)$ once one fixes matching spin structures $\mathfrak{s}$, $\mathfrak{s}'$. 
\end{proof}

\section{Lutz twists and homotopical invariants}\label{sc:homotopy_invariants_lutz_twist}

In this section, we prove Theorem \ref{prop:lutz_twist_rationally_null_homologous}, which is a direct corollary of Lemmas~\ref{lem:effect_of_contact surgery} and~\ref{lem:d_3invariant} below.

\begin{lemma}\label{lem:effect_of_contact surgery}
   Let $T$ be a transverse knot in $(M,\xi)$ and let $(M,\xi_T)$ be the overtwisted contact manifold obtained by a Lutz-twist along $T$. Then for every spin structure $\mathfrak{s}$ on $M$, we have $$\Gamma(\xi_T, \mathfrak{s})-\Gamma(\xi,\mathfrak{s})=-[T].$$
\end{lemma}

\begin{proof}
    Let $\mathcal{L}=\cup_{i=1}^nL_i$ be an oriented Legendrian link in $(S^3,\xi_{std})$ such that $(M,\xi)$ is obtained by contact $(\pm1)$-surgery along $\mathcal{L}$. By Lemma~\ref{lem:knot_in_complement} we can choose an oriented Legendrian knot $L_M$ in the complement of $\mathcal{L}$ whose negative transverse push-off is isotopic to $T$ in $(M,\xi)$. Let $L$ denotes the Legendrian knot in $(S^3,\xi_{std})$ which maps to $L_M$ under the contact $(\pm 1)$-surgeries along $\mathcal{L}$.
    By \cite{DingGeigesStipsciz}, we get a surgery diagram of $(M,\xi_T)$ by performing the same contact $(\pm1)$-surgeries on $\mathcal{L}$ followed by a pair of contact $(+1)$-surgeries on $L$ and the $2$-fold positive stabilization $L_2$ of a push-off of $L$.  Fix a spin structure $\mathfrak{s}$ on $M$ to compute the $\Gamma$-invariants from these two surgery descriptions. It is important to remember that a surgery description only determines $M$ up to contactomorphism. As such, to obtain our results, we need to always factor in the isomorphisms used to identify two different descriptions of the manifold.
    
    Let $\mu_i$ be the meridians of $L_i$ for $1\leq i \leq n$. Let $\mu_{n+1}, \mu_{n+2}$ denote the meridians of $L$ and $L_2$, respectively. Then we can write the homology class $[L_M]=\sum_{i=1}^n\alpha_i\mu_i$, for $\alpha_i \in \mathbb{Z}$, seen as a homology class in $M$ given by the surgery description along $\mathcal{L}$. 
    
    Let $\overline{\mathcal{L}}=\mathcal{L}\cup L\cup L_2$ be the contact surgery diagram for $(M,\xi_T)$ as described before and let $Q, \overline{Q}$ be the linking matrices for $\mathcal{L}$ and $\overline{\mathcal{L}}$, respectively. Let $r_i$ denote the rotation number of $L_i$, for $i=1,\ldots, n$. Let $r$ and $t$ be the rotation and Thurston--Bennequin number of $L$ in $(S^3,\xi_{std})$. At times, we will use the notation $r_{n+1}$ and $r_{n+2}$ to denote the rotation numbers of $L$ and $L_2$, respectively. Then
    \begin{align*}
    \overline{Q}=\begin{bmatrix}
       & && \alpha_1& \alpha_1\\
        &Q&&\vdots &\vdots\\
        &&&\alpha_n&\alpha_n\\
        \alpha_1&\hdots&\alpha_n&t+1&t\\
        \alpha_1&\hdots&\alpha_n&t&t-1 
    \end{bmatrix}_{(n+2)\times (n+2)}
\end{align*}

It is not hard to see, using the above matrix and elementary row-column operations, that $\mu_{n+1}+\mu_{n+2}=0$. 
In order to compute the $\Gamma$-invariant of $(M,\xi;\mathfrak{s})$, we fix a characteristic sublink $\mathcal{L}_J$ of $\mathcal{L}$ corresponding to a spin structure $\mathfrak{s}$.

\noindent \begin{claim}
    The characteristic sublink of $\overline{\mathcal{L}}$ corresponding to $\mathfrak{s}$ is the same as the sublink $\mathcal{L}_J$ of $\mathcal{L}$ if the quantity $t-lk(K,\mathcal{L}_J)$ is odd, where $t$ is the Thurston--Bennequin number of $L$. Otherwise, the characteristic sublink of $\overline{\mathcal{L}}$ corresponding to $\mathfrak{s}$ is $\mathcal{L}_J\cup L \cup L_2 $.
\end{claim} 

\begin{proof}[Proof of the claim]
    The proof of the claim follows from analyzing the sequence of Kirby moves giving the isomorphism between $(M,\mathcal{L})$ and $(M,\overline{\mathcal{L}})$ --and reversing them (see Figure \ref{figure:Kirby_moves} and Figure \ref{figure:inverse_Kirby_moves}). Using the explicit Kirby moves and the discussion in Lemma \ref{lem:sublink_modifications} regarding characteristic sublink and spin structures it is clear that the claim follows.
    First one uses the last bullet of Lemma \ref{lem:sublink_modifications} to determine when the unknot and $L$ in the top picture of Figure \ref{figure:inverse_Kirby_moves} are added to the characteristic sublink. The knot $L$ is never added to the characteristic sublink, while the unknot is added when $t+1-lk(K,\mathcal{L}_J)$ is odd -or equivalently when $t-lk(K,\mathcal{L}_J)$ is even. Then, subsequent moves as described in Figure \ref{figure:inverse_Kirby_moves} involve a handle slide -- which adds $L$ to the characteristic sublink if and only if the unknot is already in the characteristic sublink -- and isotopies, which do not affect the sublink. Thus, we see that $L$ and $L_2$ get added to characteristic sublink if and only if $t-lk(K,\mathcal{L}_J)$ is even, and neither of these gets added to the characteristic sublink if $t-lk(K,\mathcal{L}_J)$ is odd.
\end{proof}

\begin{figure}[htbp]
\centering
\includegraphics[width=0.7\linewidth]{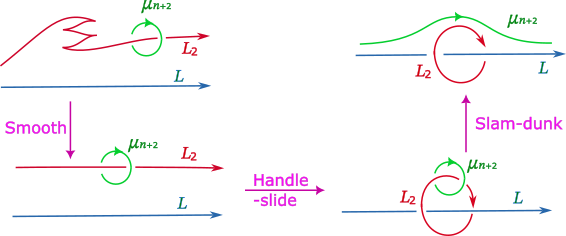}
\caption{
The sequence of Kirby moves that gives the homeomorphism between $(M,\mathcal{L})$ and $(M,\overline{\mathcal{L}})$.
}
\label{figure:Kirby_moves}
\end{figure}
\begin{figure}[htb]
\centering
\includegraphics[width=0.7\linewidth]{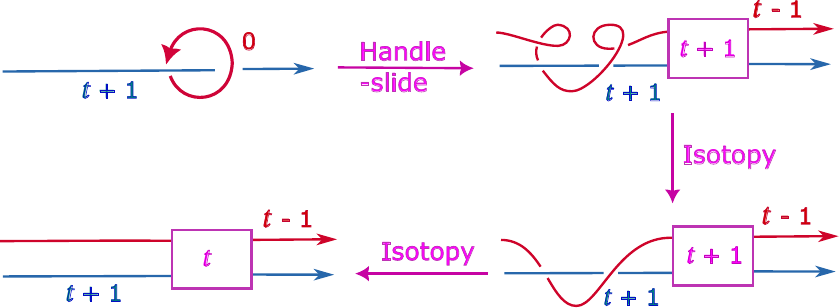}
\caption{Showing that the characteristic sublink remains unchanged while reversing the Kirby moves.}
\label{figure:inverse_Kirby_moves}
\end{figure}

Having established the relation between the characteristic sublink $\mathcal{L}_J$ of $\mathcal{L}$ corresponding to $\mathfrak{s}$ and the characteristic sublink $\overline{\mathcal{L}}_J$ of $\overline{\mathcal{L}}$ corresponding to $\mathfrak{s}$ we can now compute the $\Gamma$-invariants of the contact structures $\xi$ and $\xi_T$ using the formula from Lemma~\ref{lem:Gamma}. Recall that $\mu_{n+1}+\mu_{n+2}=0$ and $r_{n+2}=r_{n+1}+2$, we will use these identities in the computation of the $\Gamma$-invariants.

We first consider the case when $t-lk(K,\mathcal{L}_J)$ is odd. That is the characteristic sublink for $\overline{\mathcal{L}}$ is the same as the characteristic sublink $L_J$ of $\mathcal{L}$. Also note that $lk(K,\mathcal{L}_J)=\sum_{j\in J}\alpha_j$, therefore $r_{n+1}+\sum_{j\in J}\alpha_j$ is even. Thus we have
\begin{align*}
    \Gamma(\xi_T;\mathfrak{s})&=\frac{1}{2}\left\{\sum_{i=1}^{n+2} r_i\mu_i+\sum_{j\in J} (\Bar{Q}\mu)_j\right\}\\
    &=\Gamma(\xi,\mathfrak{s})+\frac{1}{2}\left\{ r_{n+1}\mu_{n+1}+(r_{n+1}+2)\mu_{n+2}+\sum_{j\in J}\alpha_j(\mu_{n+1}+\mu_{n+2})\right\}\\
    &=\Gamma(\xi,\mathfrak{s})+\mu_{n+2}+\frac{r_{n+1}+\sum_{j\in J}\alpha_j}{2}(\mu_{n+1}+\mu_{n+2})\\
    &=\Gamma(\xi,s)+\mu_{n+2}.
\end{align*}

Now let us consider the case when $t-lk(K,\mathcal{L}_J)$ is even. That is the characteristic sublink of $\overline{\mathcal{L}}$ is $\mathcal{L}_J\cup K\cup U$, where $\mathcal{L}_J$ is the characteristic sublink of $\mathcal{L}$. Then $\sum_{j\in J}\alpha_j+r_{n+1}+1$ is even and thus we have
\begin{align*}
    &\Gamma(\xi_T;\mathfrak{s})=\frac{1}{2}\left\{\sum_{i=1}^{n+2} r_i\mu_i+\sum_{j\in J} (\Bar{Q}\mu)_j+2\sum_{i=1}^n\alpha_i\mu_i+(t+1)\mu_{n+1}+t(\mu_{n+2}+\mu_{n+1})+(t-1)\mu_{n+2}\right\}\\
    &=\Gamma(\xi,\mathfrak{s})+\frac{1}{2}\left\{\sum_{j\in J}\alpha_j(\mu_{n+1}+\mu_{n+2})+2\sum_{i=1}^n\alpha_i\mu_i+(r_{n+1}+2t+1)\mu_{n+1}+(r_{n+2}+2t-1)\mu_{n+2}\right\}\\
&=\Gamma(\xi,\mathfrak{s})+\sum_{i=1}^n\alpha_i\mu_i+\frac{1}{2}\left\{(\sum_{j\in J}\alpha_j+r_{n+1}+2t+1)\right\}(\mu_{n+1}+\mu_{n+2})\\
&=\Gamma(\xi,\mathfrak{s})-(t+1)\mu_{n+1}-t\mu_{n+2}\\
&=\Gamma(\xi,\mathfrak{s})+\mu_{n+2}\\
\end{align*}
Figure \ref{figure:Kirby_moves} induces a diffeomorphism $f\colon M\rightarrow M$, that acts on homology as follows 
\begin{align*}
f_*\colon H_1(M,\mathbb{Z})&\longrightarrow H_1(M,\mathbb{Z})\\
    \mu_{i} &\longmapsto \mu_i,\: \textrm{ for all }\: 1\leq i \leq n,\\
    \mu_{n+2} &\longmapsto [L].
\end{align*}
 Therefore, we can conclude  that the difference is given by
\begin{equation*}
    f_*\big(\Gamma(\xi_T,\mathfrak{s})\big) -\Gamma(\xi,\mathfrak{s}) =f_*(\mu_{n+2})=[L]=-[T].\qedhere
\end{equation*}
\end{proof}
\begin{remark}
    Note that in the previous proof we could have chosen $L$ with $T$ as its positive transverse push-off. Then the knot $L_2$ becomes a negative 2-fold stabilization of a push-off of $L$, and the homology classes satisfy $[T]=[L]$. The computations in this case still give $\Gamma(\xi_T, \mathfrak{s})-\Gamma(\xi, \mathfrak{s)}=-[T]$. Similarly, this will apply to the proof of Lemma~\ref{lem:d_3invariant} below.
\end{remark}
We now turn our attention to the $d_3$-invariant of a contact structure for which the Euler class is torsion.

\begin{lemma}\label{lem:d_3invariant}
Let $(M,\xi)$ be a rational homology contact $3$-sphere and let $T$ be a transverse knot in $M$ with rational self-linking number $sl_{\mathbb{Q}}(T)$. Let $\xi_T$ be the contact structure on $M$ obtained by performing a Lutz twist along $T$. Then their $d_3$-invariants are related by
    \begin{align*}
        d_3(\xi_T)-d_3(\xi)=-sl_{\mathbb{Q}}(T).
\end{align*}
\end{lemma}

\begin{proof}
Let $(M,\xi)$ be a contact $3$-manifold obtained by integer surgery along an \emph{oriented} Legendrian link 
$\mathcal{L}=\cup_{i=1}^nL_i$ in $(S^3,\xi_{std})$. Let $T$ be a transverse knot in $(M,\xi)$, and let $(M,\xi_T)$ denote the contact manifold obtained by a single Lutz twist along $T$. As in the proof of Lemma \ref{lem:Gamma} we can choose an oriented Legendrian knot $L_M$ in the complement of $\mathcal{L}$ whose negative transverse push-off is isotopic to $T$ in $(M,\xi)$. Let $L$ denote the Legendrian knot in $(S^3,\xi_{std})$ which maps to $L_M$ under the contact $(\pm 1)$-surgeries along $\mathcal{L}$.
Then a surgery diagram of $(M,\xi_T)$ is given by performing the same contact $(\pm1)$-surgeries on $\mathcal{L}$ followed by a pair of contact $(+1)$-surgeries on $L$ and the $2$-fold positive stabilization $L_2$ of a push-off of $L$. Let $\overline{\mathcal{L}}=\mathcal{L}\cup L \cup L_2$. and let $\mu_i$ denote the meridians of $L_i$ for $1\leq i \leq n$ and  $\mu_{n+1}, \mu_{n+2}$ denote the meridians of $L$ and $L_2$, respectively. We write the homology class $[L]=\sum_{i=1}^n\alpha_i\mu_i$, for some  $\alpha_i$'s. We write $Q$ and $\overline{Q}$ be the linking matrices for $\mathcal{L}$ and $\overline{\mathcal{L}}$, respectively. Let $r_i$ denote the rotation number of $L_i$, for $i=1,\ldots, n$. Let $r$ and $t$ be the rotation and Thurston--Bennequin number of $L$. 

With this notational setup, we will compute the $d_3$-invariants of $\xi$ and $\xi_T$ using Lemma \ref{lem:d3}.
Since $H_1(M;\mathbb{Q})=0$, the absolute value of the determinant of the linking matrix equals the order of $H_1(M)$. Moreover, the Euler class is guaranteed to be torsion, and all the relevant objects (such as $Q^{-1}$ and $d_3$) are well-defined. 

We start by comparing the signatures and inverses of the linking matrices. 
We can write $\overline{Q}$ as
\[
\overline{Q} = \begin{bmatrix}
Q & \mathbf{\alpha} & \mathbf{\alpha} \\
\mathbf{\alpha}^\top & t+1 & t \\
\mathbf{\alpha}^\top & t & t-1
\end{bmatrix}
\]
We can rewrite \(\overline{Q}\) as a block matrix
\[
\overline{Q} = \begin{bmatrix}
Q & A \\
A^\top & B
\end{bmatrix}
\]
where
\[
A = \begin{bmatrix} \mathbf{\alpha} & \mathbf{\alpha} \end{bmatrix}, \quad B = \begin{bmatrix} t+1 & t \\ t & t-1 \end{bmatrix}.
\]
We know that $Q$ is invertible over $\mathbb{Q}$. The Schur complement of $Q$ is
\[
S = B - A^\top Q^{-1} A.
\]
We compute 
\[
A^\top Q^{-1} A = \begin{bmatrix}
\mathbf{\alpha}^\top Q^{-1} \mathbf{\alpha} & \mathbf{\alpha}^\top Q^{-1} \mathbf{\alpha} \\
\mathbf{\alpha}^\top Q^{-1} \mathbf{\alpha} & \mathbf{\alpha}^\top Q^{-1} \mathbf{\alpha}
\end{bmatrix} = d \begin{bmatrix} 1 & 1 \\ 1 & 1 \end{bmatrix}
\]
where
\[
d = \mathbf{\alpha}^\top Q^{-1} \mathbf{\alpha}
\]
Therefore
\[
S = \begin{bmatrix} t+1 & t \\ t & t-1 \end{bmatrix} - d \begin{bmatrix} 1 & 1 \\ 1 & 1 \end{bmatrix} = \begin{bmatrix} t+1 - d & t - d \\ t - d & t - 1 - d \end{bmatrix}
\]
We compute the determinant of $S$ as
\begin{align*}
\det(S) =& (t+1 - d)(t-1 - d) - (t - d)^2\\
=& (t+1)(t-1) - (t+1 + t-1)d + d^2 - (t^2 - 2td + d^2)\\
=& (t^2 - 1) - 2td + d^2 - (t^2 - 2td + d^2) = -1.
\end{align*}
Since \(S\) is \((2 \times 2)\)-matrix and \(\det(S) < 0\), it has a positive and a negative eigenvalue and thus the signatures of $Q$ and $\overline{Q}$ agree.

Next we describe the inverse of $\overline{Q}$ as
\[
\overline{Q}^{-1} = \begin{bmatrix}
Q^{-1} + Q^{-1} A S^{-1} A^\top Q^{-1} & -Q^{-1} A S^{-1} \\
-S^{-1} A^\top Q^{-1} & S^{-1}
\end{bmatrix}
\]
and compute the relevant terms as
\begin{align*}
    &S^{-1} = \frac{1}{\det(S)} \begin{bmatrix} t - 1 - d & -(t - d) \\ -(t - d) & t+1 - d \end{bmatrix} = -\begin{bmatrix} t - 1 - d & -(t - d) \\ -(t - d) & t+1 - d \end{bmatrix}
= \begin{bmatrix} d + 1 - t & t-d \\  t-d & d - t - 1 \end{bmatrix}\\
&A S^{-1} A^\top = \mathbf{\alpha} \begin{bmatrix} 1 & 1 \end{bmatrix} S^{-1} \begin{bmatrix} 1 \\ 1 \end{bmatrix} \mathbf{\alpha}^\top = \mathbf{\alpha} (1-1) \mathbf{\alpha}^\top=0
\end{align*}
If we define $\mathbf{v} = Q^{-1} \mathbf{\alpha}$  
we obtain
\[
\overline{Q}^{-1} = \begin{bmatrix}
Q^{-1}  & -\mathbf{v} \begin{bmatrix} 1& -1\end{bmatrix} \\
-\begin{bmatrix} 1 \\-1 \end{bmatrix} \mathbf{v}^\top & \begin{bmatrix}
d + 1 - t &  t-d \\
t-d & d - t - 1
\end{bmatrix}
\end{bmatrix}.
\]
We proceed to compute the $d_3$ invariant of $\xi_T$. The vector of  rotation numbers is $\mathbf{z} = (\mathbf{r}, r,r+2)$ where $\mathbf{r} = (r_1, \ldots, r_n)$.
We compute
\begin{align*}
  \mathbf{z}^\top \overline{Q}^{-1} \mathbf{z} &= \mathbf{r}^\top Q^{-1} \mathbf{r} - 2 \mathbf{r}^\top \mathbf{v} \begin{bmatrix} 1 & -1 \end{bmatrix} \begin{bmatrix} r \\ r+2 \end{bmatrix} + \begin{bmatrix} r \\ r+2 \end{bmatrix}^\top S^{-1} \begin{bmatrix} r \\ r+2 \end{bmatrix}  \\
  &=\mathbf{r}^\top Q^{-1} \mathbf{r} -2 \mathbf{r}^\top \mathbf{v} ( r + -1 (r+2)) +\mathbf{z}_2^\top S^{-1} \mathbf{z}_2\\
  &=\mathbf{r}^\top Q^{-1} \mathbf{r} +4 \mathbf{r}^\top \mathbf{v} -4r+4d-4t-4.
\end{align*}
The $d_3$-invariant of $(M,\xi_T)$ is hence given by 
$$ \frac{1}{4}\big(\mathbf{r}^\top Q^{-1} \mathbf{r} +4\mathbf{r}^\top \mathbf{v} -4r+4d-4t-4 -3\sigma(\Bar{Q}) -2\operatorname{rk}(\Bar{Q})\big)-\frac{1}{2}+(q+2), $$
where $q$ is the number of $+1$ surgeries in $\mathcal{L}$. Since $\sigma (\Bar{Q})=\sigma (Q)$ and $\operatorname{rk}(\Bar{Q})=\operatorname{rk}(Q)+2$, the difference of the $d_3$-invariants computes as
\begin{align*}
d_3(\xi_T)-d_3(\xi)=-(r+t)+\mathbf{r}^\top \mathbf{v}+d.
 \end{align*}
 
Now we write
$$ \mathbf{r}^\top \mathbf{v} + d =\langle r+\alpha, Q^{-1} \alpha\rangle.
$$
 Since $L$ (the Legendrian push-off of the transverse knot $T$) is rationally nullhomologous in $(M,\xi)$, we get $[L]=\sum_i\alpha_i\mu_i\in H_1(M)$. By \cite{Kegel-Durst} there exists a $c\in \mathbb{Z}^n$ such that $Qc=D\alpha$ for some non-zero $D\in\Z$ and
\begin{align*}
    tb_{\mathbb{Q}}(L_M)=t-\frac{1}{D}\sum_i c_i\alpha_i \,\textrm{ and }\, rot_{\mathbb{Q}}(L_M)=r-\frac{1}{D}\sum_ic_ir_i
    \end{align*}
which implies that
\begin{align*}
        \langle r+\alpha, Q^{-1} \alpha\rangle=\langle r+\alpha, \frac{1}{D}c \rangle
         = t+r-tb_{\mathbb{Q}}(L_M)-rot_{\mathbb{Q}}(L_M).
\end{align*}   
Combining these and using that $sl_\Q(T)=tb_{\mathbb{Q}}(L_M)+rot_\Q(L_M)$ we obtain the claimed formula
$$ d_3(\xi_T)-d_3(\xi)=-(t+r) + (t+r -(tb_{\mathbb{Q}}(L_M)+rot_{\mathbb{Q}}(L_M)))=-(tb_{\mathbb{Q}}(L_M)+rot_{\mathbb{Q}}(L_M)).\qedhere$$
\end{proof}

\section{Relations between \texorpdfstring{$\Gamma$}{Gamma} and \texorpdfstring{$d_3$}{d3}.}
   In this section we prove Theorem \ref{thm:d_3determinesinvariants}. For that we first state a preliminary result.
   
   \begin{lemma}(Baker--Etnyre \cite[Lemma 1.2]{baker_etnyre})\label{lem:bakr_etnyre}
       For every rationally null-homologous knot $T$ of order $D$ in $(M,\xi)$, there is a transverse null-homologous knot $T'$ in $(M,\xi)$ with $sl_{\mathbb{Q}}(T)=\frac{sl(T')}{D}$.
   \end{lemma}

   \begin{proof}
       In \cite[Lemma 1.2]{baker_etnyre} it is shown that for every rational null-homologous Legendrian knot $L'$ of order $D$ in $(M,\xi)$, there exists a null-homologous Legendrian knot $L$ in $(M,\xi)$ so that $tb_{\mathbb{Q}}(L')=\frac{tb(L)}{D}$ and $rot(L')=\frac{rot(L)}{D}$. By taking the transverse push-offs $T$ and $T'$ of $L$ and $L'$ we get the claimed statement.
   \end{proof}
   
\begin{lemma}\label{lem:not_dividing}
    Let $M$ be a rational homology sphere whose integral first homology $H_1(M;\mathbb{Z})$ only contains even torsion and let $\xi$ be a contact structure on $M$ with torsion Euler class. Let $\xi_T$ be the contact structure obtained from $\xi$ by performing a Lutz twist on a rationally null-homologous transverse knot $T$ of order $D>1$. Then the difference of the $d_3$-invariants $$ d_3(\xi_T)-d_3(\xi)$$ is not an integer.
\end{lemma}

\begin{proof}
    By combining Proposition \ref{prop:lutz_twist_rationally_null_homologous} and Lemma \ref{lem:bakr_etnyre}, we get
    $$d_3(\xi)-d_3(\xi_T)= sl_{\mathbb{Q}}(T)=\frac{sl_{\mathbb{Q}}(T')}{D}.$$ 
    Since $T$ is rationally nullhomologous and $H_1(M;\mathbb{Z})$ contains only even torsion, the order of $T$ is even. Hence, it cannot divide $sl(T')$, since it is a well known result that the self-linking number for a given null-homologous Legendrian knot $L$ in a contact manifold $M$ is always an \emph{odd integer}, see for example~\cite{Geiges}.  
 \end{proof}

\begin{theorem}\label{thm:d_3determinesinvariants_precise}
    Let $M$ be a rational homology $3$-sphere with $H_1(M;\mathbb{Z})\cong \mathbb{Z}_2$. Then the $d_3$-invariant of every contact structure $\xi$ on $M$ completely determines the $\Gamma$-invariant. More precisely: given two contact structures $\xi_1$, $\xi_2$ on $M$ then the difference $d_3(\xi_1)-d_3(\xi_2)$ is an integer if and only if $\Gamma(\xi_1;\mathfrak{s})= \Gamma(\xi_2;\mathfrak{s})$ for every choice of spin structure $\mathfrak{s}$.
\end{theorem}

   \begin{proof}
    We fix a choice of a spin structure $\mathfrak{s}$ on $M$ and express all $\Gamma$-invariants with respect to that spin structure. If $\Gamma(\xi_1;\mathfrak{s})=\Gamma(\xi_2;\mathfrak{s})$, then we know that $d_3(\xi_1)-d_3(\xi_2)\in \mathbb{Z}$, see for example~\cite{Ding_Geiges_Stipsciz_surgery_diagrams}. 

    Conversely, let us assume that $\Gamma(\xi_1;\mathfrak{s}) \neq \Gamma(\xi_2;\mathfrak{s})\in H_1(M;\Z)$. Then we choose a transverse knot $T$ in $(M,\xi_1)$ such that 
    \begin{equation*}
        [T]=\Gamma(\xi_1;\mathfrak{s})-\Gamma(\xi_2;\mathfrak{s})\neq0 \in H_1(M;\Z).
    \end{equation*}
    By Lemma~\ref{lem:not_dividing} we get that $d_3(\xi_1)-d_3(\xi_T)\notin \mathbb{Z}$. On the other hand, Proposition~\ref{prop:lutz_twist_rationally_null_homologous} implies that $\Gamma(\xi_2;\mathfrak{s})=\Gamma(\xi_T;\mathfrak{s})$ and thus $d_3(\xi_2)-d_3(\xi_T)\in \mathbb{Z}$. Thus $d_3(\xi_1)-d_3(\xi_2)$ cannot be integral.
   \end{proof}

Next, we show that on any manifold with homology not isomorphic to $0$ or $\Z_2$ the $d_3$-invariant does not determine the $\Gamma$-invariant.
  
\begin{theorem}\label{thm:prototypeforgammaandd3}
    Let $M$ be a rational homology $3$-sphere with $H_1(M,\Z)$ not isomorphic to $0$ or $\Z_2$. Then for every choice of spin structure $\mathfrak{s}$ on $M$, there exists two tangential $2$-plane fields $\xi_1$ and $\xi_2$ on such that $\Gamma(\xi_1;\mathfrak{s})\neq \Gamma(\xi_2;\mathfrak{s})$ but $d_3(\xi_1)=d_3(\xi_2)$. 
\end{theorem}

\begin{proof}
We fix a spin structure $\mathfrak{s}$ on $M$ and express all $\Gamma$-invariants with respect to this spin structure. By the classification of finitely presented abelian groups we know that
\begin{equation*}
    H_1(M;\mathbb{Z})\cong \bigoplus_{i=1}^k\mathbb{Z}_{p_i}.
\end{equation*}
Note that $M$ has the same homology as $\#_{i=1}^k L({p_i},1)$ and by Lemma~\ref{lem:homologydependenceonly} it is enough to prove the theorem in the case that $M=\#_{i=1}^k L({p_i},1)$.

 In this case $M$ can be obtained by performing $-p_i$ smooth surgeries along disjoint unknots $O_i$ in $S^3$. Let $U_i$ denote a Legendrian realization of $O_i$ with $tb(U_i)=t_i$ and $rot(U_i)=r_i$. The corresponding contact $\pm1$ surgery diagram can be obtained by considering $U_i$ such that $t_i=-p_i\pm 1$. Then the linking matrix is given by 
        \begin{align*}
            \begin{bmatrix}
                -p_1&0&\hdots&0\\
                0&-p_2&\hdots&0\\
                \vdots&\vdots&\vdots&\vdots\\
                0&0&0&-p_k
            \end{bmatrix}
        \end{align*} 
        Let $\xi$ be the contact structure on $M$ obtained by this contact $(\pm 1)$-surgery description. Let $\widetilde{T}$ be a transverse knot in $(S^3,\xi_{std})$ disjoint from the above surgery description and such that for a fixed $j$,  $lk(\widetilde{T},U_i)=0$ if $i\neq j$ and $lk(\widetilde{T},U_i)=l_j$ if $i=j$. Let $l:=(0,\hdots, l_j, 0,\hdots,0)$. Moreover $\widetilde T$ represents a transverse knot $T$ in $(M,\xi)$. Let $(M,\xi_T)$ be the overtwisted contact manifold obtained by Lutz twisting along $T$ and let $D$ be the order of $[T]$ in $M$ with $[T]=l\in H_1(M;\mathbb{Z})$. We have $D=\frac{p_j}{gcd(p_j,l_j)}$ and  $c=(0, \hdots,0,-\frac{l_j}{gcd(p_j,l_j)},0,\hdots, 0)$ provides a solution for $Qc=Dl$. Therefore using the formula from \cite{Kegel_thesis}, we have
        \begin{align*}
        sl_{\mathbb{Q}}(T)=sl(\Tilde{T})+\frac{\frac{l_j}{gcd(p_j,l_j)}}{\frac{p_j}{gcd(p_j,l_j)}}(l_j+r_j)
        =sl(\Tilde{T})+\frac{l_j}{p_j}(l_j+r_j).
        \end{align*}
        By Proposition~\ref{prop:lutz_twist_rationally_null_homologous}, the difference
        \begin{align*}
            d_3(\xi)-d_3(\xi_T)=sl(\Tilde{T})+\frac{l_j}{p_j}(l_j+r_j)
        \end{align*}
        is an integer if and only if $p_j$ divides $l_j(l_j+r_j)$. 
        
        We know that either $p_j=2$ for all $j$ or there exists a $j$ such that $p_j>2$ since $H_1(M;\mathbb{Z})\ncong \mathbb{Z}_2$, $0$. First, we assume $p_j> 2$ and our goal is to make a choice for $T$ and adjust the rotation number $r_j$ of $U_j$ fixing $t_j$ such that $\Gamma(\xi)-\Gamma(\xi_T)=[T]\neq 0$ and $p_j$ divides $(l_j+r_j)$. Note that $[T]\neq 0$ is equivalent to the condition that $p_j$ does not divides $l_j$. 

        First, we consider the case that $p_j\geq 3$ is odd. The rotation number of $U_j$ belongs to the set $\{\pm1,\hdots, \pm (p_j-2)\}$ and $\{\pm 1, \hdots, \pm p_j\}$ for $t_j=-p_j+1$ and $t_j=-p_j-1$, respectively. Now we choose $\widetilde{T}$ such that $l_j\in \{0,2,\ldots, p_j-1\}$ for $t_j=-p_j-1$ and $l_j\in \{2,4,\ldots, p_j-1\}$ for $t_j=-p_j+1$. Then there is also a Legendrian realization of $U_j$ with $r_j=p_j-l_j$ for such a choice of $\widetilde{T}$ and its linking number $l_j$ with $U_j$. Then $p_j$ does not divide $l_j$ but $p_j$ divides $(l_j+r_j)$. 

        Secondly, we consider the case that $p_j\geq 3$ is even. Then for $t_j=-p_j+1$ the rotation number of $U_j$ can take values in the set $\{0,\pm 2, \hdots, \pm (p_j-2)\}$ and for $t_j=-p_j-1$, the rotation number takes values in the set $\{0,\pm 2, \hdots, \pm p_j\}$. In both the cases we can make choice for $U_j$ such that $r_j\neq 0, \pm p_j$ and a choice for $\widetilde{T}$ such that $l_j=p_j-r_j$. Then $p_j$ does not divide $l_j$ but $p_j$ divides $(l_j+r_j)$. 
        
        In both cases, we have that $\Gamma(\xi)-\Gamma(\xi_T)=[T]\neq 0$ and $d_3(\xi_T)-d_3(\xi)$ is an integer. By changing the self-linking number of $\widetilde T$ in $(S^3,\xi_{st})$ by performing connected sums with appropriate transverse knots, we can assume $d_3(\xi_T)=d_3(\xi)$.
    
     We are left with the case, where $p_i=2$ for all $i =1,\ldots k\geq2$. In that case we choose a contact structure $\xi$ on $M$ and consider two transverse knots $T_1$ and $T_2$ in $(M,\xi)$ such that $[T]=(1,0,\hdots,0)$ and $[T_2]=(0,1,0,\hdots,0)$. We know that $\Gamma(\xi_{T_1})-\Gamma(\xi_{T_2})=[T_2]-[T_1]\neq 0$. By Lemma \ref{lem:bakr_etnyre}, there exists nullhomologous transverse knots $T'_1$ and $T'_2$ in $(M,\xi)$ such that $sl_{\mathbb{Q}}(T_i)=\frac{sl(T_i)}{2}, \: i=1,2$. Thus by proposition~\ref{prop:lutz_twist_rationally_null_homologous} we have
        \begin{align*}
            d_3(\xi_{T_1})-d_3(\xi_{T_2})=&sl_\mathbb{Q}(T_2)-sl_{\mathbb{Q}}(T_1)
        =\frac{sl(T'_2)-sl(T'_1)}{2}\in \mathbb{Z}
        \end{align*}
since the self-linking number is always an odd integer. Again by adjusting the self-linking number of, say $T_1$, appropriately we can arrange that $d_3(\xi_{T_1})=d_3(\xi_{T_2})$.
\end{proof}

  \begin{proof}[Proof of Theorem \ref{thm:d_3determinesinvariants}]
 If $H_1(M;\mathbb{Z})=0$ then the $\Gamma$-invariant is always vanishing and thus determined by $d_3$. If $H_1(M;\mathbb{Z})$ is isomorphic to $\Z_2$ then $\Gamma$ is determined by $d_3$ as shown in Theorem \ref{thm:d_3determinesinvariants_precise}. Finally, in Theorem~\ref{thm:prototypeforgammaandd3} we have shown that $\Gamma$ is not determined by $d_3$ on all other rational homology $3$-spheres.
\end{proof}
   
\section{Upper bounds on contact surgery distances}\label{sc:surgery_distance}

In this section, we establish our results for bounds on the contact surgery distance, i.e.\ we prove Theorem~\ref{thm:contact_surgery_distance} and its generalisations. We start by giving the precise definitions of the various surgery distances.
The following definition can be found, for example, in \cite{Ichihara-Saito} and generalises the surgery number of \cite{Auckly}.

\begin{definition}\label{def:surgery_distance}
    Let $M$ and $N$ be closed, connected, oriented $3$-manifolds. Then the \textit{surgery distance} of $M$ and $N$, denoted as $s(M,N)$, is defined as the minimum number of rational Dehn surgeries required to obtain $N$ from $M$. The \textit{integeral surgery distance} $S_\Z(M,N)$ of $M$ and $N$ is defined as the minimum number of integral Dehn surgeries required to obtain $N$ from $M$.
\end{definition}

\begin{remark} We add a few remarks about these definitions.
\begin{enumerate}
    \item First we remark that the surgeries are performed along knots in general $3$-mani\-folds, and thus will in general not be null-homologous. Thus, there is no preferred longitude to measure the surgery coefficients. However, the integrality of a surgery coefficient is still well-defined since this is independent of the choice of the longitude. Indeed, let $K$ be a knot in a $3$-manifold $M$ with meridian $\mu$ and two longitudes $\lambda$ and $\lambda'$. Then these two longitudes are related by $\lambda=n\mu+\lambda'$, for some integer $n\in\Z$. A surgery coefficient $p/q$ measured with respect to the longitude $\lambda$ can be expressed with respect to the longitude $\lambda'$ as follows
    \begin{align*}
        p/q\equiv p\mu+q\lambda= (p+n q)\mu +q\lambda' \equiv p/q + n 
    \end{align*}
    and thus the integrality of a surgery coefficient is independent of the choice of the longitude.
    \item It is not hard to check that if $N$ is obtained by a single Dehn surgery along a knot $K$ in $M$, then also $M$ can be obtained by a single Dehn surgery along the dual knot $K^*$ from $N$. Moreover, if the surgery along $K$ was integral, then the surgery along $K^*$ is also integral. Thus, it follows from the Likorish--Wallace theorem that both versions of the surgery distance are well-defined and are symmetric. 
\end{enumerate}
\end{remark}

Next, we define the analogous notions for contact surgery. Recall that a Legendrian knot $K$ in a contact $3$-manifold $(M,\xi)$ admits a preferred framing coming from the contact structure, and we always measure surgery coefficients with respect to this framing.

\begin{definition}\label{def:contact_surgery_distance}
    The \textit{contact $(\pm1)$-surgery distance}, $cs_{\pm1}((M,\xi),(N,\xi'))$ between $(M,\xi)$ and $(N,\xi')$ is defined to be the minimum number of contact $(\pm1)$-surgeries that are required to obtain $(N,\xi')$ from $(M,\xi)$.
    
    The \textit{contact surgery distance}, $cs((M,\xi),(N,\xi'))$ from $(M,\xi)$ to $(N,\xi')$ is defined to be the minimum number of rational contact surgeries (with non-vanishing surgery coefficient) that are required to obtain $(N,\xi')$ from $(M,\xi)$.

    Here we say that $(M_k,\xi_k)$ is obtained by $k$ contact surgeries from $(M_0,\xi_0)$ if there exists a sequence of contact manifolds $(M_i,\xi_i)$, where $(M_i,\xi_i)$ is obtained from $(M_{i-1},\xi_{i-1})$ by a single contact surgery along a Legendrian knot in $(M_{i-1},\xi_{i-1})$ for all $i=1,\ldots, k$.
\end{definition}

\begin{remark}
    We add a few remarks about these definitions.
    \begin{enumerate}
        \item First, we remark that one could define other interesting notions of contact surgery distances by restricting the contact surgery coefficients to be integers or reciprocals of integers. But in this article, we consider only the above two notions.
        \item The different notions of contact surgery distances are well-defined by the Ding--Geiges theorem, and since the cancellation lemma implies that a contact $(\pm1)$-surgery can be reversed by a contact $(\mp1)$-surgery.
        \item By Lemma~\ref{lem:knot_in_complement} every Legendrian knot in a contact manifold can be presented in any contact $(\pm1)$-surgery diagram of that contact manifold, it follows that Definition~\ref{def:contact_surgery_distance} is equivalent to the one given in the introduction, i.e.\ $cs_{\pm1}((M,\xi),(N,\xi'))$ is equal to the minimal number of components of a Legendrian link $L$ in $(M,xi)$ such that $(N,\xi')$ arises by contact $(\pm1)$-surgery along $L$. On the other hand, one cannot reverse a contact surgery with a general contact surgery coefficient by a single contact surgery, see for example Example 4.7.2 in~\cite{Kegel_thesis}. So the above alternative description remains unclear for the general contact surgery distance.
        \item Similarly, the cancellation lemma implies that the contact $(\pm1)$-surgery distance is a metric. But this is not true for the general contact surgery distance as the following result shows.
    \end{enumerate}
\end{remark}

\begin{proposition}
Any contact structure on $S^3$ with rational contact surgery number $cs=1$ that is not isotopic to $\xi_1$, the unique overtwisted contact structure on $S^3$ with $d_3=1$, has contact surgery distance to $(S^3,\xi_{std})$ bigger than one. In particular, the contact surgery distance is not symmetric.
\end{proposition}

\begin{proof}
    Let $\xi'$ be an overtwisted contact structure on $S^3$ with contact surgery number one. If $(S^3,\xi_{std})$ is obtained from $(S^3,\xi')$ by a single contact surgery along a Legendrian knot $K$ in $(S^3,\xi')$, then $K$ has to be an unknot by the resolution of the knot complement problem~\cite{Gordon_Luecke}. Moreover, since $(S^3,\xi_{std})$ is tight the exterior of $K$ has to be tight and thus $K$ is a non-loose Legendrian unknot in $(S^3,\xi')$. But the classification of non-loose Legendrian unknots in contact structures on $S^3$~\cite{Eliashberg_Fraser} implies that $\xi'$ is isotopic to $\xi_1$.

    By \cite{KegelEtnyreOnaran}, there exist infinitely many contact structures on $S^3$ with contact surgery number one, and thus the contact surgery distance is not symmetric.
\end{proof}

Now we are ready to state and prove the generalization of Theorem \ref{thm:contact_surgery_distance}.
    
\begin{theorem}\label{thm:maingeneral}
Let $(M,\xi)$ and $(N,\xi')$ be two contact, oriented contact $3$-manifolds. Then 
\begin{align*}
    cs_{\pm1}((M,\xi),(N,\xi'))\leq s_{\mathbb{Z}}(M,N)+5\\
    cs((M,\xi),(N,\xi'))\leq s(M,N)+5.
\end{align*}
\end{theorem}

\begin{proof}
Let $\mathcal{K}$ be a smooth link in $M$ with $s(M,N)$-com\-po\-nents (resp. $s_{\mathbb{Z}}(M,N)$-compo\-nents) such that smooth surgery (resp. integer surgery) along $\mathcal{K}$ yields $N$. 

First we perform a contact $(\pm1)$-surgery along a Legendrian knot in $(M,\xi)$ to obtain an overtwisted contact structure $\xi_{ot}$ on $M$. (For example one can perform a contact $(+1)$-surgery on a once stabilizied Legendrian unknot in a Darboux ball.) Next we choose a loose Legendrian realization $\mathcal L$ of $\mathcal K$ in $(M,\xi_{ot})$. Since there is an overtwisted disk in the complement of $\mathcal{L}$, we can stabilize or take connected sums of the overtwisted disk to realize any contact framing on each component of $\mathcal L$. Thus we can assume that contact surgery with non-zero surgery coefficient (with $(\pm1)$ surgery coefficient) on $\mathcal L$ yields an overtwisted contact structure $\xi'_{ot}$ on $N$.

We choose a transverse knot $T$ inside $(N,\xi')$ such that for some choice of spin structure $\mathfrak{s}$ on $N$ we have
\begin{align*}
    [T]=\Gamma(\xi',\mathfrak s)-\Gamma(\xi'_{ot},\mathfrak s) \in H_1(N;\Z).
\end{align*}
By Lemma~\ref{lem:effect_of_contact surgery} the overtwisted contact structure $\xi'_T$ obtained by a Lutz twist on $T$ has the same $\Gamma$-invariant as $\xi'_{ot}$. By \cite{DingGeigesStipsciz} this Lutz twist can be realized by two contact $(+1)$-surgeries.

Now any two overtwisted contact structures on the same $3$-manifold that have equal $\Gamma$-invariants are related by taking a connected sum with an overtwisted contact structure on $S^3$, see for example~\cite{Ding_Geiges_Stipsciz_surgery_diagrams}. Since any overtwisted contact structure on $S^3$ can be obtained from $(S^3,\xi_{std})$ by at most two contact $(\pm1)$-surgeries~\cite{KegelEtnyreOnaran} it follows that the surgery distance of $(N,\xi'_{ot})$ and $(N,\xi'_T)$ is at most two.

So in total we have performed one contact surgery to make $(M,\xi)$ overtwisted, $s_{\mathbb{Z}}(M,N)$ or $s(M,N)$ many contact surgeries to get $(N,\xi_{ot}')$, two more contact surgeries to reach $(N,\xi'_T)$, and two more contact surgeries to undo the Lutz twist.
\end{proof}

\begin{remark}
    We also remark that the above proof also shows that if $(M,\xi)$ is already overtwisted then the upper bounds can be reduced by one. 
\end{remark}

\section{Contact surgery distances of lens spaces} \label{sec:lens_spaces}

The contact surgery distance between two contact manifolds is bounded from below by the topological surgery distance of the underlying smooth manifolds. However, the topological surgery distance is a quantity that is difficult to compute in general. Sometimes it is possible to compute lower bounds for the contact surgery distance even if the topological surgery distance is not known. We demonstrate this below for the case of tight contact structures on lens spaces.

For coprime integers $p$ and $q$ define a lens space $L(p,q)$ to be the $3$-manifold obtained by $(-p/q)$-surgery on the unknot. Thus, lens spaces are exactly the $3$-manifolds that can be obtained by gluing together two solid tori. We conclude that any two distinct lens spaces have rational surgery distance $1$. From the classification of tight contact structures on lens spaces~\cite{Giroux_lens,Honda} it follows that any two lens spaces also have rational contact surgery distance one.

On the other hand, the integral surgery distance between lens spaces remains unclear. For example, it is not known if there is a lens space that cannot be obtained by two integral surgeries from $S^3$. However, there is the following simple homological obstruction for two lens spaces having integral surgery distance one.

\begin{proposition}\label{prop:top_dist_lens}
    If there exists an integral surgery from the lens space $L(p,q)$ to the lens space $L(p',q')$, then at least one of $p'q$ or $-p'q$ is a square $\mod p$ and at least one of $pq'$ or $-pq'$ is a square mod $ p'$.

    In particular, for any given prime $p$ with $p\equiv 1 \mod 4$ and any integer $p'$, we find a $q$ such that for every $q'$ coprime to $p'$ there exists no integral surgery from $L(p,q)$ to $L(p',q')$.
\end{proposition}

Note that for $p'=1$ we recover Proposition 1 of~\cite{Fintushel_Stern}. For the proof of Proposition~\ref{prop:top_dist_lens} we need a bit of background on quadratic reciprocity, which we recall first.

For an odd prime $p$, we define the \textit{Legendre symbol} 
\begin{align*}
    \left(\frac{q}{p}\right):=\begin{cases} 1  &\text{if }q\equiv l^2 \mod p \text{ for some integer }l,\\ -1  &\text{if }q\nequiv l^2 \mod p \text{ for any integer }l.\end{cases}
\end{align*}
We have the following computation rules
\begin{align*}
    &\left(\frac{-1}{p}\right)=\begin{cases} 1  &\text{if }p\equiv 1 \mod 4 ,\\ -1  &\text{if }p\equiv 3 \mod 4 \end{cases},\\
    &\left(\frac{q}{p}\right)\left(\frac{q'}{p}\right)=\left(\frac{qq'}{p}\right),\\
    &\sum_{q=1}^{p-1}\left(\frac{q}{p}\right)=0.
\end{align*}
The last equation says that half of the values of $q$ are squares mod $p$ and the other half are not.

We also recall the classification of lens spaces. Two lens spaces $L(p,q)$ and $L(p,q')$ are orientation-preservingly diffeomorphic if and only if $q\equiv q' \mod p$ or $q'$ is the multiplicative inverse of $q$ mod $p$. In particular, we can always consider $q$ mod $p$ and by possibly changing the orientation and excluding $S^3$ and $S^1\times S^2$ we can assume $0<q<p$.

\begin{proof}[Proof of Proposition~\ref{prop:top_dist_lens}]
    We write $q^*$ for the multiplicative inverse of $q$ mod $p$. By the classification of lens spaces we have that $L(p,q)$ is orientation-preservingly diffeomorphic to $L(p,q^*)$. If there is an integeral surgery from $L(p,q)$ to $L(p',q')$ then we can find an integer $n$ and a knot $K$ in the complement of the unknot $U$ in $S^3$ such that $(-p/{q^*})$-surgery on $U$ followed by $n$-surgery on $K$ yields $L(p',q')$. The homology of the latter manifold is isomorphic to $\Z_{p'}$. On the other hand we can compute the homology of the surgery description to have order
    \begin{equation*}
        \left\vert \det \begin{pmatrix}
-p & l\\
q^*l & n
\end{pmatrix}\right\vert=pn+q^*l^2,
    \end{equation*}
    where $l$ denotes the linking number of $U$ and $K$. Thus it follows that $\pm p'\equiv q^*l^2\mod p$ and thus $\pm p'q\equiv l^2\mod p$. If we reverse the roles of $L(p,q)$ and $L(p',q')$ we get that also $\pm pq'$ is a square mod $p'$.

    For the second statement, we choose $q$ such that 
    \begin{equation*}
        \left(\frac{p'q}{p}\right)=\left(\frac{p'}{p}\right)\left(\frac{q}{p}\right)=-1,
    \end{equation*}
    which is possible since half the values of $q$ are squares mod $p$ and the other half are not. Then we compute that also
    \begin{equation*}
        \left(\frac{-p'q}{p}\right)=\left(\frac{-1}{p}\right)(-1)=-1,
    \end{equation*}
    since $p\equiv 1 \mod 4$. Thus, both $p'q$ and $-p'q$ are not squares mod $ p$ and we can apply the first part of the proposition.
\end{proof}

In summary, we can easily construct infinitely many pairs of lens spaces with integer surgery distance larger than one. Nevertheless in general the integral surgery distance of lens spaces remains unclear. However, using the extra rigidity from tight contact structures, we are able to demonstrate the existence of infinitely many pairs of tight contact lens spaces that have contact $(\pm1)$-surgery distance larger than one. We point out that the obstruction for the integer surgery distance from Proposition~\ref{prop:top_dist_lens} does not apply in these cases.

\begin{theorem} \label{thm:lens_spaces_lower_bound}
    There exist infinitely many pairs of lens spaces $L(p,q)$ and $L(p',q')$ that are not obstructed from being related by a single integer surgery via Proposition~\ref{prop:top_dist_lens}, but such that for all tight contact structures $\xi$ and $\xi'$ on $L(p,q)$ and $L(p',q')$, we have
    $$cs_{\pm1}\big((L(p,q),\xi),(L(p',q'),\xi')\big)>1.$$
\end{theorem}

\begin{proof}
We use a criterion of Plamenevskaya \cite{Plamenevskaya}, which says that we cannot obtain a tight $(L(p',q'),\xi')$ from a tight $(L(p,q),\xi)$ by a single Legendrian surgery (normalized such that $0<q<p$ and $0<q'<p'$) if all of the following three conditions are fulfilled
\begin{itemize}
    \item $p$ and $p'$ are coprime,
    \item  $-pq'$ is a square mod $p'$, and
    \item $-1$ is not a square mod $p'$.
\end{itemize}
By the cancellation lemma, it follows that two tight lens spaces $(L(p,q),\xi)$ and $(L(p',q'),\xi')$ have $cs_{\pm1}>1$ if all of the following conditions are fulfilled
\begin{enumerate}
    \item $p$ and $p'$ are coprime,
    \item  $-pq'$ is a square mod $p'$, 
    \item  $-p'q$ is a square mod $p$, 
    \item $-1$ is not a square mod $p'$, and
    \item $-1$ is not a square mod $p$.
\end{enumerate}
Note that the second and third conditions show directly that Proposition~\ref{prop:top_dist_lens} does not apply to show that the integer surgery distance is larger than one. So it remains to demonstrate that there are infinitely many pairs of coprime integers $(p,q)$ and $(p',q')$ with $0<q<p$ and $0<q'<p'$ that fulfill Conditions $(1)$--$(5)$. 

For that, let $p$ and $p'$ be different primes with $p\equiv p'\equiv 3 \mod 4$. Then Condition $(1)$ is automatically fulfilled. Since $p\equiv 3 \mod 4$, we also have $\left(\frac{-1}{p}\right)=-1$ and thus $-1$ is not a square mod $p$. This concludes Condition $(5)$, and the same reasoning applies to conclude Condition $(4)$. 

Since half of the values of $q$ are squares mod $p$ and the other half are not, we can choose $q$ and $q'$ such that 
\begin{equation*}
   \left(\frac{q}{p}\right)=-\left(\frac{p'}{p}\right) \,\text{ and }\, \left(\frac{q'}{p'}\right)=-\left(\frac{p}{p'}\right).
\end{equation*}
Then we compute that
    \begin{equation*}
   \left(\frac{-p'q}{p}\right)=\left(\frac{-1}{p}\right) \left(\frac{p'}{p}\right) \left(\frac{q}{p}\right) =(-1)\left(\frac{p'}{p}\right) (-1)\left(\frac{p'}{p}\right)=(\pm1)^2=1 
\end{equation*}
and thus $-p'q$ is a square mod $p$. The same applies to show that Condition $(2)$ holds.

Finally, we note that Dirichlet's theorem implies that there are infinitely many primes $p$ with $p\equiv 3 \mod 4$, which gives us infinitely many pairs of such lens spaces.
\end{proof}

\let\MRhref\undefined
  \bibliographystyle{hamsalpha}
	\bibliography{bibliography.bib}
	
\end{document}